\documentclass[a4paper,11pt]{article}
\usepackage{amsmath,amsthm,amsfonts,amssymb,graphicx,mathtools,iwona}
\usepackage[colorlinks=true,citecolor=black,linkcolor=black,urlcolor=blue]{hyperref}
\usepackage[tmargin=35mm,bmargin=35mm,lmargin=28.5mm,rmargin=28.5mm]{geometry}
\usepackage{rsfso}
\usepackage[nodayofweek]{datetime}
\usepackage[noabbrev,capitalise]{cleveref}
\usepackage[numbers,sort&compress]{natbib}
\makeatletter
\def\NAT@spacechar{~}
\makeatother
\theoremstyle{plain}
\newtheorem{thm}{Theorem}
\newtheorem{lem}[thm]{Lemma}
\newtheorem{ques}[thm]{Question}
\crefname{lem}{Lemma}{Lemmas}
\crefname{thm}{Theorem}{Theorems}
\newcommand{\arXiv}[1]{arXiv:\,\href{http://arxiv.org/abs/#1}{#1}}
\newcommand{\msn}[1]{MR:\,\href{http://www.ams.org/mathscinet-getitem?mr=MR#1}{#1}}
\newcommand{\doi}[1]{doi:\,\href{http://dx.doi.org/#1}{#1}}

\renewcommand{\baselinestretch}{1.2}
\usepackage[hang]{footmisc}

\renewcommand{\thefootnote}{\fnsymbol{footnote}}	
\allowdisplaybreaks
\newcommand\DateFootnote{
\begingroup
\renewcommand\thefootnote{}
\footnote{17$^{th}$ July, 2017,; revised \today}
\setcounter{footnote}{0}
\vspace*{-3ex}
\endgroup}
\makeatletter
\renewcommand\section{\@startsection {section}{1}{\z@}{-3ex \@plus -1ex \@minus -.2ex}{2ex \@plus.2ex}{\normalfont\large\bfseries}}
\renewcommand\subsection{\@startsection{subsection}{2}{\z@}{-2.5ex\@plus -1ex \@minus -.2ex}{1.5ex \@plus .2ex}{\normalfont\normalsize\bfseries}}
\renewcommand\subsubsection{\@startsection{subsubsection}{3}{\z@}{-2ex\@plus -1ex \@minus -.2ex}{1ex \@plus .2ex}{\normalfont\normalsize\bfseries}}
 \renewcommand\paragraph{\@startsection{paragraph}{4}{\z@}{1.5ex \@plus.5ex \@minus.2ex}{-1em}{\normalfont\normalsize\bfseries}}
\renewcommand\subparagraph{\@startsection{subparagraph}{5}{\parindent}  {1.5ex \@plus.5ex \@minus .2ex}  {-1em} {\normalfont\normalsize\bfseries}}
\makeatother
\renewcommand{\thefootnote}{\fnsymbol{footnote}}	
\newcommand{\PP}{\mathcal{P}}
\newcommand{\RP}[1]{\mathcal{P}\!\langle#1\rangle}

\newcommand{\CC}{\mathcal{C}}

\DeclarePairedDelimiter\floor\lfloor\rfloor
\renewcommand{\geq}{\geqslant}
\renewcommand{\leq}{\leqslant}

\begin{document}

{\Large\bfseries\boldmath\scshape Bad News for Chordal Partitions}

\medskip
Alex Scott\footnotemark[3] \quad
Paul Seymour\footnotemark[4] \quad
David R. Wood\footnotemark[5]

\DateFootnote

\footnotetext[3]{Mathematical Institute, University of Oxford, Oxford, U.K.\ (\texttt{scott@maths.ox.ac.uk}).}

\footnotetext[4]{Department of Mathematics, Princeton University, New Jersey, U.S.A. (\texttt{pds@math.princeton.edu}). Supported by ONR grant N00014-14-1-0084 and NSF grant DMS-1265563.}

\footnotetext[5]{School of Mathematical Sciences, Monash University, Melbourne, Australia\\ (\texttt{david.wood@monash.edu}). Supported by the Australian Research Council.}

\emph{Abstract.} Reed and Seymour [1998] asked whether every graph has a partition into induced connected non-empty bipartite subgraphs such that the quotient graph is chordal. If true, this would have significant ramifications for Hadwiger's Conjecture. We prove that the answer is `no'. In fact, we show that the answer is still `no' for several relaxations of the question.

\bigskip

\medskip

\hrule

\bigskip

\renewcommand{\thefootnote}{\arabic{footnote}}	
\section{Introduction}

Hadwiger's Conjecture \citep{Hadwiger43} says that for all $t\geq 0$ every graph with no $K_{t+1}$-minor is $t$-colourable. This conjecture is easy for $t\leq 3$, is equivalent to the 4-colour theorem for $t=4$, is true for $t=5$ \citep{RST-Comb93}, and is open for $t\geq 6$. The best known upper bound on the chromatic number is $O(t\sqrt{\log t})$, independently due to \citet{Kostochka82,Kostochka84} and \citet{Thomason84,Thomason01}. This conjecture is widely considered to be one of the most important open problems in graph theory; see \citep{SeymourHC} for a survey. 

Throughout this paper, we employ standard graph-theoretic definitions (see \citep{Diestel4}), with one important exception: we say that a graph $G$ \emph{contains} a graph $H$ if $H$ is isomorphic to an induced subgraph of $G$. 
 
Motivated by Hadwiger's Conjecture, \citet{ReedSeymour-JCTB98} introduced the following definitions\footnote{\citet{ReedSeymour-JCTB98} used different terminology: `chordal decomposition' instead of chordal partition, and `touching pattern' instead of quotient.}. A \emph{vertex-partition}, or simply \emph{partition}, of a graph $G$ is a set $\PP$ of non-empty induced subgraphs of $G$ such that each vertex of $G$ is in exactly one element of $\PP$. Each element of $\PP$ is called a \emph{part}. The \emph{quotient} of $\PP$ is the graph, denoted by $G/\PP$, with vertex set $\PP$ where distinct parts $P,Q\in \PP$ are adjacent in $G/\PP$ if and only if some vertex in $P$ is adjacent in $G$ to some vertex in $Q$. A partition of $G$ is \emph{connected} if each part is connected. We (almost) only consider connected partitions. In this case, the quotient is the minor of $G$ obtained by contracting each part into a single vertex. A  partition is \emph{chordal} if it is connected and the quotient is chordal (that is, contains no induced cycle of length at least four). Every graph has a chordal partition (with a 1-vertex quotient). Chordal partitions are a useful tool when studying graphs $G$ with no $K_{t+1}$ minor. Then for every connected partition $\PP$ of $G$, the quotient $G/\PP$ contains no $K_{t+1}$, so if in addition $\PP$ is chordal, then $G/\PP$ is $t$-colourable (since chordal graphs are perfect).  \citet{ReedSeymour-JCTB98} asked the following  question (repeated in \citep{KawaReed08,SeymourHC}). 

\begin{ques} 
\label{ChordalPartition}
Does every graph have a chordal partition such that each part is bipartite?
\end{ques} 

If true, this would imply that every graph with no $K_{t+1}$-minor is $2t$-colourable, by taking the product of the $t$-colouring of the quotient with the 2-colouring of each part. This would be a major breakthrough for Hadwiger's Conjecture. The purpose of this note is to answer Reed and Seymour's question in the negative. In fact, we show the following stronger result. 

\begin{thm}
\label{NoChordal}
For every integer $k\geq 1$ there is a graph $G$, such that for every chordal partition $\PP$ of $G$, some part of $\PP$ contains $K_k$. 
Moreover, for every integer $t\geq 4$ there is a graph $G$ with tree-width at most $t-1$ (and thus with no $K_{t+1}$-minor) such that 
for every chordal partition $\PP$ of $G$, some part of $\PP$ contains a complete graph on at least $\floor{(3t-11)^{1/3}}$ vertices. 
\end{thm}

\cref{NoChordal} says that it is not possible to find a chordal partition in which each part has bounded chromatic number. What if we work with a larger class of partitions? The following natural class arises. A partition of a graph is \emph{perfect} if it is connected and the quotient graph is perfect. If $\PP$ is a perfect partition of a $K_{t+1}$-minor free graph $G$, then $G/\PP$ contains no $K_{t+1}$ and is therefore $t$-colourable.  So if every part of $\PP$ has small chromatic number, then we can control the chromatic number of $G$. We are led to the following relaxation of \cref{ChordalPartition}: does every graph have a perfect partition in which every part has bounded chromatic number?  Unfortunately, this is not the case.

\begin{thm}
\label{NoPerfect}
For every integer $k\geq 1$ there is a graph $G$, such that for every perfect partition $\PP$ of $G$, some part of $\PP$ contains $K_k$. 
Moreover, for every integer $t\geq 6 $ there is a graph $G$ with tree-width at most $t-1$ (and thus with no $K_{t+1}$-minor), 
such that for every perfect partition $\PP$ of $G$, some part of $\PP$ contains a complete graph on at least $\floor{(\frac{3}{2}t-8)^{1/3}}$
 vertices. 
\end{thm}

\cref{NoChordal,NoPerfect} say that it is hopeless to improve on the $O(t\sqrt{\log t})$ bound for the chromatic number of $K_t$-minor-free graphs using chordal or perfect partitions directly. Indeed, the best possible upper bound on the chromatic number using the above approach would be $O(t^{4/3})$ (since the quotient is $t$-colourable, and the best possible upper bound on the chromatic number of the parts would be $O(t^{1/3})$.)\ 

What about using an even larger class of partitions?  Chordal graphs contain no  4-cycle, and perfect graphs contain no 5-cycle. These are the only properties of chordal and perfect graphs used in the proofs of \cref{NoChordal,NoPerfect}. Thus the following result is a qualitative generalisation of both \cref{NoChordal,NoPerfect}. It says that there is no hereditary class of graphs for which the above colouring strategy works.

\begin{thm}
\label{General}
For every integer $k\geq 1$ and graph $H$, there is a graph $G$, such that for every connected partition $\PP$ of $G$, either some part of $\PP$ contains $K_k$ or the quotient $G/\PP$ contains $H$. 
\end{thm}


Before presenting the proofs, we mention some applications of chordal partitions and related topics. Chordal partitions have proven to be a useful tool in the study of the following topics for $K_{t+1}$-minor-free graphs: cops and robbers pursuit games \citep{Andreae86},  fractional colouring \citep{ReedSeymour-JCTB98,KawaReed08}, generalised colouring numbers \citep{HOQRS17}, and defective and clustered colouring \citep{vdHW}. These papers show that every graph with no $K_{t+1}$ minor has a chordal partition in which each part has desirable properties. For example, in \citep{ReedSeymour-JCTB98}, each part has a stable set on at least half the vertices, and in \citep{vdHW}, each part has maximum degree $O(t)$ and is 2-colourable with monochromatic components of size $O(t)$. 

Several papers \citep{DMW05,KP-DM08,Wood-JGT06} have shown that graphs with tree-width $k$ have chordal partitions in which the quotient is a tree, and each part induces a subgraph with tree-width $k-1$, amongst other properties. Such partitions have been used for queue and track layouts \citep{DMW05} and non-repetitive graph colouring \citep{KP-DM08}. A \emph{tree partition} is a (not necessarily connected) partition of a graph whose quotient is a tree; these have also been widely studied \citep{Edenbrandt86,Halin91,Seese85,Wood09,DO95,DO96,BodEng-JAlg97,Bodlaender-DMTCS99}. Here the goal is to have few vertices in each part of the partition. For example, a referee of \citep{DO95} proved that every graph with tree-width $k$ and maximum degree $\Delta$ has a tree partition with $O(k\Delta)$ vertices in each part. 

\section{Chordal Partitions: Proof of \cref{NoChordal}}

Let $\PP=\{P_1,\dots,P_m\}$ be a partition of a graph $G$, and let $X$ be an induced subgraph of $G$. Then the \emph{restriction} of $\PP$ to $X$ is the partition of $X$ defined by $$\RP{X} := \{G[V(P_i)\cap V(X)]:i\in\{1,\dots,m\},V(P_i)\cap V(X)\neq\emptyset\}.$$ Note that the restriction of a connected partition to a subgraph need not be connected. The following lemma gives a scenario where the restriction is connected. 

\begin{lem}
\label{InducedPartition}
Let $X$ be an induced subgraph of a graph $G$, such that the neighbourhood of each component of $G-V(X)$ is a clique (in $X$). 
Let $\PP$ be a connected partition of $G$ with quotient $G/\PP$. Then $\RP{X}$ is a connected partition of $X$, 
and the quotient of $\RP{X}$ is the subgraph of $G/\PP$ induced by those parts that intersect $X$. 
\end{lem}

\begin{proof}
We first prove that for every connected subgraph $G'$ of $G$, if $V(G')\cap V(X)\neq\emptyset$, then $G'[V(G')\cap V(X)]$ is connected. Consider non-empty sets $A,B$ that partition $V(G')\cap V(X)$. Let $P$ be a shortest path from $A$ to $B$ in $G'$. Then no internal vertex of $P$ is in $V(X)$. If $P$ has an internal vertex, then all its interior belongs to one component $C$ of $G-V(X)$, implying the endpoints of $P$ are in the neighbourhood of $C$ and are therefore adjacent, a contradiction. Thus $P$ has no interior, and hence $G'[V(G')\cap V(X)]$ is connected. 

Apply this observation with each part of $\PP$ as $G'$. It follows that $\RP{X}$ is a connected partition of $X$.  Moreover, if adjacent parts $P$ and $Q$ of $\PP$ both intersect $X$, then by the above observation with $G'=G[V(P)\cup V(Q)]$, there is an edge between $V(P)\cap V(X)$ and $V(Q)\cap V(X)$. Conversely, if there is an edge between $V(P)\cap V(X)$ and $V(Q)\cap V(X)$ for some parts $P$ and $Q$ of $\PP$, then $PQ$ is an edge of $G/\PP$. Thus the quotient of $\RP{X}$ is the subgraph of $G/\PP$ induced by those parts that intersect $X$. 
\end{proof}

The next lemma with $r=1$ implies \cref{NoChordal}. To obtain the second part of \cref{NoChordal} apply \cref{ChordalWork} with $k=\floor{(3t-11)^{1/3}}$, in which case $s(k,1)\leq t$. 

\begin{lem}
\label{ChordalWork}
For all integers $k\geq 1$ and $r\geq 1$, if 
$$s(k,r):=\tfrac13(k^3-k) +(r-1)k + 4 , $$
then there is a graph $G(k,r)$ with tree-width at most $s(k,r)-1$ (and thus with no $K_{s(k,r)+1}$-minor),   
such that for every chordal partition $\PP$ of $G$, either:\\
(1) $G$ contains a $K_{kr}$ subgraph intersecting each of $r$ distinct parts of $\PP$ in $k$ vertices, or\\
(2) some part of $\PP$  contains $K_{k+1}$.
\end{lem}

\begin{proof}
Note that $s(k,r)$ is the upper bound on the size of the bags in the tree-decomposition of $G(k,r)$ that we construct. We proceed by induction on $k$ and then $r$. 
When $k=r=1$, the graph with one vertex satisfies (1) for every chordal partition and has a tree-decomposition with one bag of size $1<s(1,1)$.

First we prove that the $(k,1)$ and $(k,r)$ cases imply the $(k,r+1)$ case. Let $A:=G(k,1)$ and $B:=G(k,r)$. Let $G$ be obtained from $A$ as follows. For each $k$-clique $C$ in $A$, add a  copy $B_C$ of $B$  (disjoint from the current graph), where $C$ is complete to $B_C$. We claim that $G$ has the claimed properties of $G(k,r+1)$.

By assumption, in every chordal partition of $A$ some part contains $K_k$, $A$ has a tree-decomposition with bags of size at most $s(k,1)$, and for each $k$-clique $C$ in $A$, there is a tree-decomposition of $B_C$ with bags of  size at most $s(k,r)$. For every tree-decomposition of a graph and for each clique $C$, there is a bag containing $C$. Add an edge between the node corresponding to a bag containing $C$ in the tree-decomposition of $A$ and any node of the tree in the tree-decomposition of $B_C$, and add $C$ to every bag of the tree-decomposition of $B_C$. We obtain a tree-decomposition of $G$ with bags of  size at most $\max\{s(k,1),s(k,r)+k\} = s(k,r)+k = s(k,r+1)$, as desired. 

Consider a chordal partition $\PP$ of $G$. By \cref{InducedPartition}, $\RP{A}$ is a connected partition of $A$, and the quotient of $\RP{A}$ equals the subgraph of $G/\PP$ induced by those parts that intersect $A$. Since $G/\PP$ is chordal, the quotient of  $\RP{A}$ is chordal. Since $A=G(k,1)$, by induction, $\RP{A}$ satisfies (1) with $r=1$ or (2). If outcome (2) holds, then some part of $\PP$ contains $K_{k+1}$ and outcome (2) holds for $G$. 

Now assume that $\RP{A}$ satisfies outcome (1) with $r=1$; that is, some part $P$ of $\PP$ contains some $k$-clique $C$ of $A$. If some vertex of $B_C$ is in $P$, then $P$ contains $K_{k+1}$ and outcome (2) holds for $G$. Now assume that no vertex of $B_C$ is in $P$. Since each part of $\PP$ is connected, the parts of $\PP$ that intersect $B_C$ do not intersect $G-V(B_C)$. Thus, $\RP{B_C}$ is a connected partition of $B_C$, and the quotient of $\RP{B_C}$ equals the subgraph of $G/\PP$ induced by those parts that intersect $B_C$, and is therefore chordal. Since $B=G(k,r)$, by induction, $\RP{B_C}$ satisfies (1) or (2). If outcome (2)  holds, then the same outcome holds for $G$. Now assume that outcome (1) holds for $B_C$. Thus $B_C$ contains a $K_{kr}$ subgraph intersecting each of $r$ distinct parts of $\PP$ in $k$ vertices. None of these parts are $P$. Since $C$ is complete to $B_C$,  $G$ contains a $K_{k(r+1)}$ subgraph intersecting each of $r+1$ distinct parts of $\PP$ in $k$ vertices, and outcome (1) holds for $G$. Hence $G$ has the claimed properties of $G(k,r+1)$. 


It remains to prove the $(k,1)$ case for $k\geq 2$. By induction, we may assume the $(k-1,k+1)$ case. Let $A:=G(k-1,k+1)$. As illustrated in \cref{FirstConstruction}, let $G$ be obtained from $A$ as follows: for each set $\CC=\{C_1,\dots,C_{k+1}\}$ of pairwise-disjoint $(k-1)$-cliques in $A$, whose union induces $K_{(k-1)(k+1)}$, add a $K_{k+1}$ subgraph $B_{\CC}$ (disjoint from the current graph), whose $i$-th vertex is adjacent to  every vertex in $C_i$. We claim that $G$ has the claimed properties of $G(k,1)$.

\begin{figure}
\centering
\includegraphics{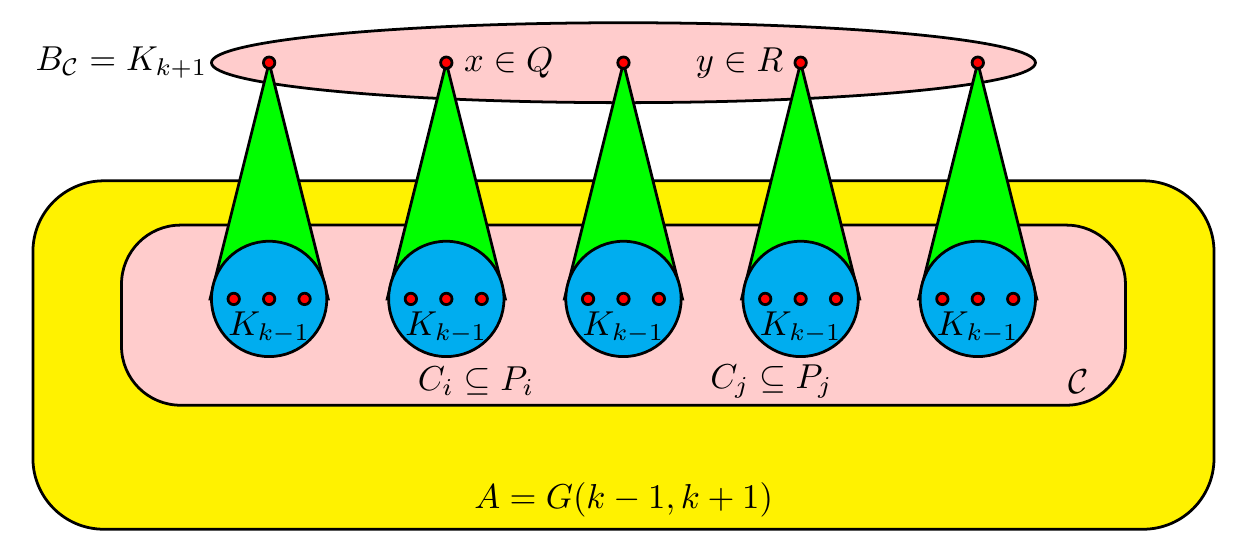}
\caption{Construction of $G(k,1)$ in \cref{ChordalWork}. 
\label{FirstConstruction}}
\end{figure}

By assumption, $A$ has a tree-decomposition with bags of size at most $s(k-1,k+1)$. For each set $\CC=\{C_1,\dots,C_{k+1}\}$ of pairwise-disjoint $(k-1)$-cliques in $A$, whose union induces $K_{(k-1)(k+1)}$, choose a node $x$ corresponding to a bag of the tree-decomposition of $A$ containing $C_1\cup\dots\cup C_{k+1}$, and add a new node adjacent to $x$ with corresponding bag $V(B_{\CC})\cup C_1\cup\dots\cup C_{k+1}$. We obtain a tree-decomposition of $G$ with bags of  size at most $\max\{s(k-1,k+1),(k+1)k\} = s(k-1,k+1)=s(k,1)$, as desired. 

Consider a chordal partition $\PP$ of $G$. By \cref{InducedPartition}, $\RP{A}$ is a connected partition of $A$ and the quotient of $\RP{A}$ equals the subgraph of $G/\PP$ induced by those parts that intersect $A$, and is therefore chordal. Since $A=G(k-1,k+1)$, by induction,  $\RP{A}$ satisfies (1) or (2). If outcome (2) holds for $\RP{A}$, then some part of $\PP$ contains $K_{k}$ and outcome (1) holds for $G$ (with $r=1$). Now assume that outcome (1) holds for $\RP{A}$. Thus $A$ contains a $K_{(k-1)(k+1)}$ subgraph intersecting each of $k+1$ distinct parts $P_1,\dots,P_{k+1}$ of $\PP$ in $k-1$ vertices. Let $C_i$ be the corresponding $(k-1)$-clique in $P_i$. Let $\CC:=\{C_1,\dots,C_{k+1}\}$ and $\widehat{\CC}:= C_1\cup\dots\cup C_{k+1}$. 

If for some $i\in\{1,\dots,k+1\}$, the neighbour of $C_i$ in $B_{\CC}$ is in $P_i$, 
then $P_i$ contains $K_{k}$ and outcome (1) holds for $G$. 
Now assume that for each $i\in\{1,\dots,k+1\}$, the neighbour of $C_i$ in $B_{\CC}$ is not in $P_i$. 
Suppose that some vertex $x$ in $B_{\CC}$ is in $P_i$ for some $i\in\{1,\dots,k+1\}$. 
Then since $P_i$ is connected, there is a path in $G$ 
between $C_i$ and $x$ avoiding the neighbourhood of $C_i$ in $B_{\CC}$. 
Every such path intersects $\widehat{\CC}\setminus C_i$, but none of these vertices are in $P_i$. 
Thus, no vertex in  $B_{\CC}$ is in $P_1\cup\dots\cup P_{k+1}$. 
If $B_{\CC}$ is contained in one part, then outcome (2) holds. 
Now assume that there are  vertices $x$ and $y$ of $B_{\CC}$  in distinct parts $Q$ and $R$ of $\PP$. 
Then $x$ is adjacent to every vertex in $C_i$ and $y$ is adjacent to every vertex in $C_j$, for some distinct $i,j\in\{1,\dots,k+1\}$. 
Observe that $(Q,R,P_j,P_i)$ is a 4-cycle in $G/\PP$. 
Moreover, there is no $QP_j$ edge in $G/\PP$
because $(\widehat{\CC}\setminus C_j)\cup\{y\}$ separates $x\in Q$ from $C_j\subseteq P_j$, 
and none of these vertices are in $Q\cup P_j$. 
Similarly, there is no $RP_i$ edge in $G/\PP$.
Hence $(Q,R,P_j,P_i)$ is an induced 4-cycle in $G/\PP$, 
which contradicts the assumption that $\PP$ is a chordal partition. 
Therefore $G$ has the claimed properties of $G(k,1)$. 
\end{proof}

\section{Perfect Partitions: Proof of \cref{NoPerfect}}

The following lemma with $r=1$  implies \cref{NoPerfect}. 
To obtain the second part of \cref{NoPerfect} apply \cref{ChordalWork} with $k=\floor{(\frac{3}{2}t-8)^{1/3}}$, in which case $t(k,1)\leq t$. 
The proof is very similar to \cref{ChordalWork} except that we force $C_5$ in the quotient instead of $C_4$. 


\begin{lem}
\label{NoPerfectLemma}
For all integers $k\geq 1$ and $r\geq 1$, if 
$$t(k,r):=\tfrac23 (k^3-k) + (r-1)k + 6, $$
then there is a graph $G(k,r)$ with tree-width at most $t(k,r)-1$ (and thus with no $K_{t(k,r)+1}$-minor),  
such that for every perfect partition $\PP$ of $G$, either:\\
(1) $G$ contains a $K_{kr}$ subgraph intersecting each of $r$ distinct parts of $\PP$ in $k$ vertices, or\\
(2) some part of $\PP$ contains $K_{k+1}$.
\end{lem}

\begin{proof}
Note that $t(k,r)$ is the upper bound on the size of the bags in the tree-decomposition of $G(k,r)$ that we construct. 
We proceed by induction on $k$ and then $r$. 
For the base case, the graph with one vertex satisfies (1) for $k=r=1$ and has a tree-decomposition with one bag of size $1< t(1,1)$. The proof that the $(k,1)$ and $(k,r)$ cases imply the $(k,r+1)$ case is identical to the analogous step in the proof in \cref{ChordalWork}, so we omit it. 

It remains to prove the $(k,1)$ case for $k\geq 2$. By induction, we may assume the $(k-1,2k+1)$ case. 
Let $A:=G(k-1,2k+1)$. Let $B$ be the graph consisting of two copies of $K_{k+1}$ with one vertex in common. Note that $|V(B)|=2k+1$. As illustrated in \cref{SecondConstruction}, let $G$ be obtained from $A$ as follows: for each set $\CC=\{C_1,\dots,C_{2k+1}\}$ of pairwise-disjoint $(k-1)$-cliques in $A$, whose union induces $K_{(k-1)(2k+1)}$, add a subgraph $B_{\CC}$ isomorphic to $B$ (disjoint from the current graph), whose $i$-th vertex is adjacent to every vertex in $C_i$. We claim that $G$ has the claimed properties of $G(k,1)$.

\begin{figure}
\centering
\includegraphics{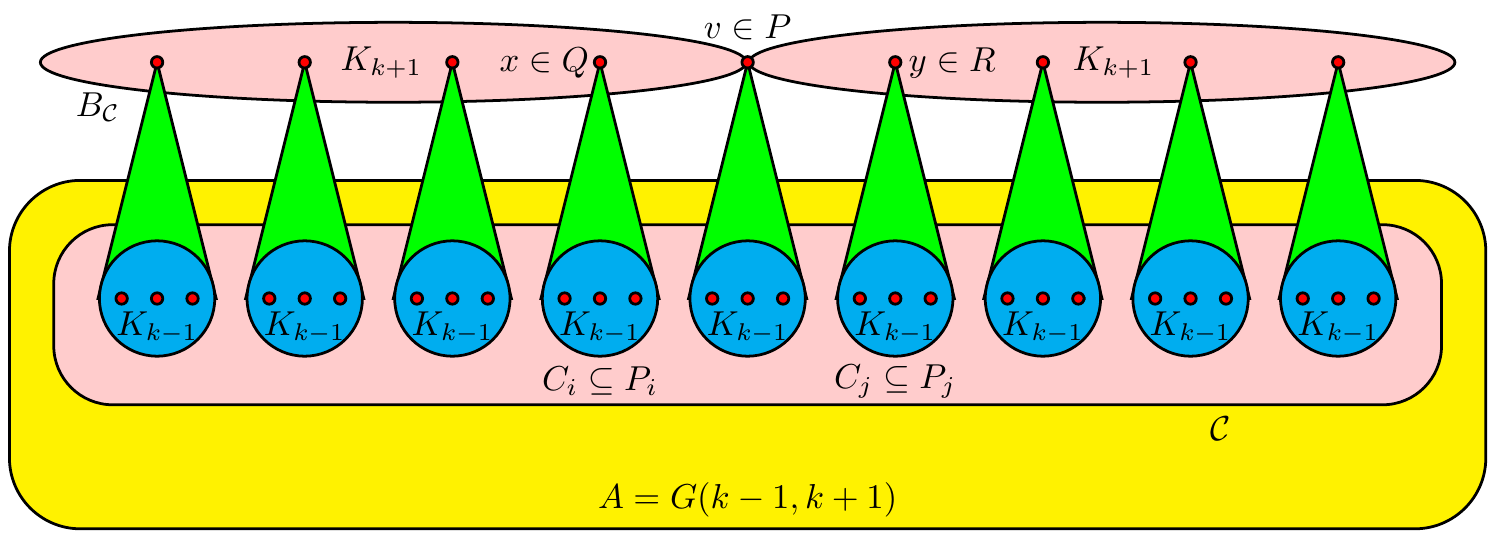}
\caption{Construction of $G(k,1)$ in \cref{NoPerfectLemma}. 
\label{SecondConstruction}}
\end{figure}

By assumption, $A$ has a tree-decomposition with bags of size at most $t(k-1,2k+1)$. For each set $\CC=\{C_1,\dots,C_{2k+1}\}$ of pairwise-disjoint $(k-1)$-cliques in $A$, whose union induces $K_{(k-1)(2k+1)}$, choose a node $x$ corresponding to a bag containing $C_1\cup\dots\cup C_{2k+1}$ in the tree-decomposition of $A$, and add a new node adjacent to $x$ with corresponding bag $V(B_{\CC})\cup C_1\cup\dots\cup C_{2k+1}$. We obtain a tree-decomposition of $G$ with bags of  size at most  $\max\{t(k-1,2k+1),(2k+1)k\} = t(k-1,2k+1)=t(k,1)$, as desired. 

%

Consider a perfect partition $\PP$ of $G$. By \cref{InducedPartition}, $\RP{A}$ is a connected partition of $A$ and the quotient of $\RP{A}$ equals the subgraph of $G/\PP$ induced by those parts that intersect $A$, and is therefore perfect. Recall that $A=G(k-1,2k+1)$. If outcome (2) holds for $\RP{A}$, then some part of $\PP$ contains $K_{k}$ and outcome (1) holds for $G$ (with $r=1$). Now assume that outcome (1) holds for $\RP{A}$. Thus $A$ contains a $K_{(k-1)(2k+1)}$ subgraph intersecting each of $2k+1$ distinct parts $P_1,\dots,P_{2k+1}$ of $\PP$ in $k-1$ vertices. Let $C_i$ be the corresponding $(k-1)$-clique in $P_i$. Let $\CC:=\{C_1,\dots,C_{2k+1}\}$ and $\widehat{\CC}:= C_1\cup\dots\cup C_{2k+1}$. 

If for some $i\in\{1,\dots,2k+1\}$, the neighbour of $C_i$ in $B_{\CC}$ is in $P_i$, 
then $P_i$ contains a $K_{k}$ subgraph and outcome (1) holds for $G$. 
Now assume that for each $i\in\{1,\dots,2k+1\}$, the neighbour of $C_i$ in $B_{\CC}$ is not in $P_i$. 
Suppose that some vertex $x$ in $B_{\CC}$ is in $P_i$ for some $i\in\{1,\dots,k+1\}$. 
Then since $P_i$ is connected, there is a path in $G$ between $C_i$ and $x$ avoiding the neighbourhood of $C_i$ in $B_{\CC}$. 
Every such path intersects $\widehat{\CC}\setminus C_i$, but none of these vertices are in $P_i$. 
Thus, no vertex in  $B_{\CC}$ is in $P_1\cup\dots\cup P_{2k+1}$. 

By construction,  $B_{\CC}$ consists of two $(k+1)$-cliques $B^1$ and $B^2$, intersecting in one vertex $v$. 
Say $v$ is in part $P$ of $\PP$. 
If $B^1 \subseteq V(P)$, then outcome (2) holds. 
Now assume that there is a vertex $x$ of $B^1$ in some part $Q$ distinct from $P$. 
Similarly, assume that there is a vertex $y$ of $B^2$ in some part $R$ distinct from $P$. 
Now, $Q\neq R$, since $\widehat{\CC}\cup\{v\}$ separates $x$ and $y$, 
and none of these vertices are in $Q\cup R$. 
By construction, $x$ is adjacent to every vertex in $C_i$ and $y$ is adjacent to every vertex in $C_j$, for some distinct $i,j\in\{1,\dots,2k+1\}$. 
Observe that $(Q,P,R,P_j,P_i)$ is a 5-cycle in $G/\PP$. 
Moreover, there is no $QP_j$ edge in $G/\PP$
because $(\widehat{\CC}\setminus C_j)\cup\{y\}$ 
separates $x\in Q$ from $C_j\subseteq P_j$, and none of these vertices are in $Q\cup P_j$. 
Similarly, there is no $RP_i$ edge in $G/\PP$.
There is no $PP_j$ edge in $G/\PP$
because $(\widehat{\CC}\setminus C_j)\cup\{y\}$ 
separates $v\in P$ from $C_j\subseteq P_j$, and none of these vertices are in $P\cup P_j$. 
Similarly, there is no $PP_i$ edge in $G/\PP$.  
Hence $(Q,P,R,P_j,P_i)$ is an induced 5-cycle in $G/\PP$, 
which contradicts the assumption that $\PP$ is a perfect partition. 
Therefore $G$ has the claimed properties of $G(k,1)$. 
\end{proof}

\section{General Partitions: Proof of \cref{General}}

To prove \cref{General} we show the following stronger result, in which $G$ only depends on $|V(H)|$.

\begin{lem}
\label{GeneralGeneral}
For all integers $k,t,r\geq 1$, there is a graph $G=G(k,t,r)$, such that for every connected partition $\PP$ of $G$ either:\\
(1) $G$ contains a $K_{kr}$ subgraph intersecting each of $r$ distinct parts of $\PP$ in $k$ vertices, or\\
(2) $G/\PP$ contains every $t$-vertex graph, or\\
(3) some part of $\PP$  contains $K_{k+1}$. 
\end{lem}

\begin{proof}
We proceed by induction on $k+t$ and then $r$. We first deal with two base cases. 
First suppose that $t=1$. Let $G:=G(k,1,r):=K_1$. Then for every partition $\PP$ of $G$, the quotient $G/\PP$ has at least one vertex, and (2) holds. Now assume that $t\geq 2$. 
Now suppose that $k=1$. Let $G:=G(1,t,r):=K_r$. Then for every connected partition $\PP$ of $G$, if some part of $\PP$  contains an edge, then (3) holds; otherwise each part is a single vertex, and (1) holds. 
Now assume that $k\geq 2$. 

The proof that the $(k,t,1)$ and $(k,t,r)$ cases imply the $(k,t,r+1)$ case is identical to the analogous step in the proof in \cref{ChordalWork}, so we omit it. 

It remains to prove the $(k,t,1)$ case for $k\geq 2$ and $t\geq 2$. 
By induction, we may assume the $(k,t-1,1)$ case and the $(k-1,t,r)$ case for all $r$. 
Let $B:=G(k,t-1,1)$ and $n:=|V(B)|$. 
Let $S^1,\dots,S^{2^n}$ be the distinct subsets of $V(B)$. 
Let $A:=G(k-1,t,2^n)$. 
Let $G$ be obtained from $A$ as follows: 
for each set $\CC=\{C_1,\dots,C_{2^n}\}$ of pairwise-disjoint $(k-1)$-cliques in $A$, whose union induces $K_{(k-1)2^n}$, 
add a copy $B_{\CC}$ of $B$  (disjoint from the current graph), where $C_i$ is complete to $S_{\CC}^i$ for all $i\in\{1,\dots,2^n\}$, 
where we write $S_{\CC}^i$ for the subset of $V(B_{\CC})$ corresponding to $S^i$.
We claim that $G$ has the claimed properties of $G(k,t,1)$.

Consider a connected partition $\PP$ of $G$. By \cref{InducedPartition}, $\RP{A}$ is a connected partition of $A$, and the quotient of $\RP{A}$ equals the subgraph of $G/\PP$ induced by those parts that intersect $A$. Recall that $A=G(k-1,t,2^n)$. If  $\RP{A}$ satisfies outcome (2), then the quotient of $\RP{A}$ contains every $t$-vertex graph and outcome (2) is satisfied for $G$. If outcome (3) holds for $\RP{A}$, then some part of $\PP$ contains $K_{k}$ and outcome (1) holds for $G$ (with $r=1$). Now assume that outcome (1) holds for $\RP{A}$. Thus $A$ contains a $K_{(k-1)2^n}$ subgraph intersecting each of $2^n$ distinct parts $P_1,\dots,P_{2^n}$ of $\PP$ in $k-1$ vertices. Let $C_i$ be the corresponding $(k-1)$-clique in $P_i$. Let $\CC:=\{C_1,\dots,C_{2^n}\}$. 

If for some $i\in\{1,\dots,2^n\}$, some neighbour of $C_i$ in $B_{\CC}$ is in $P_i$, 
then $P_i$ contains $K_k$ and outcome (1) holds for $G$. 
Now assume that for each $i\in\{1,\dots,2^n\}$, no neighbour of $C_i$ in $B_{\CC}$ is in $P_i$. 
Suppose that some vertex $x$ in $B_{\CC}$ is in $P_i$ for some $i\in\{1,\dots,2^n\}$. 
Then since $P_i$ is connected, $G$ contains a path between $C_i$ and $x$ avoiding the neighbourhood of $C_i$ in $B_{\CC}$. 
Every such path intersects $C_1\cup\dots\cup C_{i-1}\cup C_{i+1}\cup \dots\cup C_{2^n}$, but none of these vertices are in $P_i$. 
Thus, no vertex in  $B_{\CC}$ is in $P_1\cup\dots\cup P_{2^n}$. 
Hence, no part of $\PP$ contains vertices in both $B_{\CC}$ and in the remainder of $G$. 
Therefore, $\RP{B_{\CC}}$ is a connected partition of $B_{\CC}$, and the quotient of $\RP{B_{\CC}}$ equals the subgraph of $G/\PP$ induced by those parts that intersect $B_{\CC}$. 
Since $B=G(k,t-1,1)$, by induction, $\RP{B_{\CC}}$ satisfies (1), (2) or (3). 
If outcome (1) or (3) holds for $\RP{B_{\CC}}$, then the same outcome holds for $G$. 
Now assume that outcome (2) holds for $\RP{B_{\CC}}$. 

We now show that outcome (2) holds for $G$. 
Let $H$ be a $t$-vertex graph, let $v$ be a vertex of $H$, and let $N_H(v)=\{w_1,\dots,w_d\}$. 
Since outcome (2) holds for  $\RP{B_{\CC}}$, 
the quotient of $\RP{B_{\CC}}$ contains $H-v$. 
Let $Q_1,\dots,Q_{d}$ be the parts corresponding to $w_1,\dots,w_d$. 
Then $S_{\CC}^i=V(Q_1\cup \dots\cup Q_{d})$ for some $i\in\{1,\dots,2^n\}$. 
In $G/\PP$, the vertex corresponding to $P_i$ is adjacent to $Q_1,\dots,Q_d$ and to no other vertices corresponding to parts contained in $B_\CC$. Thus, including $P_i$, $G/\PP$ contains $H$ and outcome (2) holds for $\PP$. 
Hence $G$ has the claimed properties of $G(k,t,1)$. 
\end{proof}

  \let\oldthebibliography=\thebibliography
  \let\endoldthebibliography=\endthebibliography
  \renewenvironment{thebibliography}[1]{%
    \begin{oldthebibliography}{#1}%
      \setlength{\parskip}{0ex}%
      \setlength{\itemsep}{0ex}%
  }%
  {%
    \end{oldthebibliography}%
  }


\begin{thebibliography}{24}
\providecommand{\natexlab}[1]{#1}
\providecommand{\url}[1]{\texttt{#1}}
\providecommand{\urlprefix}{}
\expandafter\ifx\csname urlstyle\endcsname\relax
  \providecommand{\doi}[1]{doi:\discretionary{}{}{}#1}\else
  \providecommand{\doi}{doi:\discretionary{}{}{}\begingroup
  \urlstyle{rm}\Url}\fi

\bibitem[{Andreae(1986)}]{Andreae86}
\textsc{Thomas Andreae}.
\newblock On a pursuit game played on graphs for which a minor is excluded.
\newblock \emph{J. Comb. Theory, Ser. {B}}, 41(1):37--47, 1986.
\newblock \doi{10.1016/0095-8956(86)90026-2}.
\newblock \msn{0854602}.

\bibitem[{Bodlaender(1999)}]{Bodlaender-DMTCS99}
\textsc{Hans~L. Bodlaender}.
\newblock A note on domino treewidth.
\newblock \emph{Discrete Math. Theor. Comput. Sci.}, 3(4):141--150, 1999.
\newblock \urlprefix\url{https://dmtcs.episciences.org/256}.
\newblock \msn{1691565}.

\bibitem[{Bodlaender and Engelfriet(1997)}]{BodEng-JAlg97}
\textsc{Hans~L. Bodlaender and Joost Engelfriet}.
\newblock Domino treewidth.
\newblock \emph{J. Algorithms}, 24(1):94--123, 1997.
\newblock \doi{10.1006/jagm.1996.0854}.
\newblock \msn{1453952}.

\bibitem[{Diestel(2010)}]{Diestel4}
\textsc{Reinhard Diestel}.
\newblock \emph{Graph theory}, vol. 173 of \emph{Graduate Texts in
  Mathematics}.
\newblock Springer, 4th edn., 2010.
\newblock \urlprefix\url{http://diestel-graph-theory.com/}.
\newblock \msn{2744811}.

\bibitem[{Ding and Oporowski(1995)}]{DO95}
\textsc{Guoli Ding and Bogdan Oporowski}.
\newblock Some results on tree decomposition of graphs.
\newblock \emph{J. Graph Theory}, 20(4):481--499, 1995.
\newblock \doi{10.1002/jgt.3190200412}.
\newblock \msn{1358539}.

\bibitem[{Ding and Oporowski(1996)}]{DO96}
\textsc{Guoli Ding and Bogdan Oporowski}.
\newblock On tree-partitions of graphs.
\newblock \emph{Discrete Math.}, 149(1--3):45--58, 1996.
\newblock \doi{10.1016/0012-365X(94)00337-I}.
\newblock \msn{1375097}.

\bibitem[{Dujmovi{\'c} et~al.(2005)Dujmovi{\'c}, Morin, and Wood}]{DMW05}
\textsc{Vida Dujmovi{\'c}, Pat Morin, and David~R. Wood}.
\newblock Layout of graphs with bounded tree-width.
\newblock \emph{SIAM J. Comput.}, 34(3):553--579, 2005.
\newblock \doi{10.1137/S0097539702416141}.
\newblock \msn{2137079}.

\bibitem[{Edenbrandt(1986)}]{Edenbrandt86}
\textsc{Anders Edenbrandt}.
\newblock Quotient tree partitioning of undirected graphs.
\newblock \emph{BIT}, 26(2):148--155, 1986.
\newblock \doi{10.1007/BF01933740}.
\newblock \msn{0840323}.

\bibitem[{Hadwiger(1943)}]{Hadwiger43}
\textsc{Hugo Hadwiger}.
\newblock \"{U}ber eine {K}lassifikation der {S}treckenkomplexe.
\newblock \emph{Vierteljschr. Naturforsch. Ges. Z\"urich}, 88:133--142, 1943.
\newblock \msn{0012237}.

\bibitem[{Halin(1991)}]{Halin91}
\textsc{Rudolf Halin}.
\newblock Tree-partitions of infinite graphs.
\newblock \emph{Discrete Math.}, 97:203--217, 1991.
\newblock \doi{10.1016/0012-365X(91)90436-6}.
\newblock \msn{1140802}.

\bibitem[{van~den Heuvel et~al.(2017)van~den Heuvel, Ossona~de Mendez, Quiroz,
  Rabinovich, and Siebertz}]{HOQRS17}
\textsc{Jan van~den Heuvel, Patrice Ossona~de Mendez, Daniel Quiroz, Roman
  Rabinovich, and Sebastian Siebertz}.
\newblock On the generalised colouring numbers of graphs that exclude a fixed
  minor.
\newblock \emph{European J. Combinatorics}, 66:129--144, 2017.
\newblock \doi{10.1016/j.ejc.2017.06.019}.

\bibitem[{van~den Heuvel and Wood(2018)}]{vdHW}
\textsc{Jan van~den Heuvel and David~R. Wood}.
\newblock Improper colourings inspired by {H}adwiger's conjecture.
\newblock \emph{J. London Math. Soc.}, 2018.
\newblock \doi{10.1112/jlms.12127}.
\newblock \arXiv{1704.06536}.

\bibitem[{Kawarabayashi and Reed(2008)}]{KawaReed08}
\textsc{Ken{-}ichi Kawarabayashi and Bruce~A. Reed}.
\newblock Fractional coloring and the odd {H}adwiger's conjecture.
\newblock \emph{European J. Comb.}, 29(2):411--417, 2008.
\newblock \doi{10.1016/j.ejc.2007.02.010}.
\newblock \msn{2388377}.

\bibitem[{Kostochka(1982)}]{Kostochka82}
\textsc{Alexandr~V. Kostochka}.
\newblock The minimum {H}adwiger number for graphs with a given mean degree of
  vertices.
\newblock \emph{Metody Diskret. Analiz.}, 38:37--58, 1982.
\newblock \msn{0713722}.

\bibitem[{Kostochka(1984)}]{Kostochka84}
\textsc{Alexandr~V. Kostochka}.
\newblock Lower bound of the {H}adwiger number of graphs by their average
  degree.
\newblock \emph{Combinatorica}, 4(4):307--316, 1984.
\newblock \doi{10.1007/BF02579141}.
\newblock \msn{0779891}.

\bibitem[{K{\"u}ndgen and Pelsmajer(2008)}]{KP-DM08}
\textsc{Andr\'e K{\"u}ndgen and Michael~J. Pelsmajer}.
\newblock Nonrepetitive colorings of graphs of bounded tree-width.
\newblock \emph{Discrete Math.}, 308(19):4473--4478, 2008.
\newblock \doi{10.1016/j.disc.2007.08.043}.
\newblock \msn{2433774}.

\bibitem[{Reed and Seymour(1998)}]{ReedSeymour-JCTB98}
\textsc{Bruce~A. Reed and Paul Seymour}.
\newblock Fractional colouring and {H}adwiger's conjecture.
\newblock \emph{J. Combin. Theory Ser. B}, 74(2):147--152, 1998.
\newblock \doi{10.1006/jctb.1998.1835}.
\newblock \msn{1654153}.

\bibitem[{Robertson et~al.(1993)Robertson, Seymour, and Thomas}]{RST-Comb93}
\textsc{Neil Robertson, Paul Seymour, and Robin Thomas}.
\newblock Hadwiger's conjecture for ${K}\sb 6$-free graphs.
\newblock \emph{Combinatorica}, 13(3):279--361, 1993.
\newblock \doi{10.1007/BF01202354}.
\newblock \msn{1238823}.

\bibitem[{Seese(1985)}]{Seese85}
\textsc{Detlef Seese}.
\newblock Tree-partite graphs and the complexity of algorithms.
\newblock In \textsc{Lothar Budach}, ed., \emph{Proc. International Conf. on
  Fundamentals of Computation Theory}, vol. 199 of \emph{Lecture Notes in
  Comput. Sci.}, pp. 412--421. Springer, 1985.
\newblock \doi{10.1007/BFb0028825}.
\newblock \msn{0821258}.

\bibitem[{Seymour(2015)}]{SeymourHC}
\textsc{Paul Seymour}.
\newblock Hadwiger's conjecture.
\newblock In \textsc{John Forbes~Nash Jr. and Michael~Th. Rassias}, eds.,
  \emph{Open Problems in Mathematics}, pp. 417--437. Springer, 2015.
\newblock \doi{10.1007/978-3-319-32162-2}.
\newblock \msn{MR3526944}.

\bibitem[{Thomason(1984)}]{Thomason84}
\textsc{Andrew Thomason}.
\newblock An extremal function for contractions of graphs.
\newblock \emph{Math. Proc. Cambridge Philos. Soc.}, 95(2):261--265, 1984.
\newblock \doi{10.1017/S0305004100061521}.
\newblock \msn{0735367}.

\bibitem[{Thomason(2001)}]{Thomason01}
\textsc{Andrew Thomason}.
\newblock The extremal function for complete minors.
\newblock \emph{J. Combin. Theory Ser. B}, 81(2):318--338, 2001.
\newblock \doi{10.1006/jctb.2000.2013}.
\newblock \msn{1814910}.

\bibitem[{Wood(2006)}]{Wood-JGT06}
\textsc{David~R. Wood}.
\newblock Vertex partitions of chordal graphs.
\newblock \emph{J. Graph Theory}, 53(2):167--172, 2006.
\newblock \doi{10.1002/jgt.20183}.
\newblock \msn{MR2255760}.

\bibitem[{Wood(2009)}]{Wood09}
\textsc{David~R. Wood}.
\newblock On tree-partition-width.
\newblock \emph{European J. Combin.}, 30(5):1245--1253, 2009.
\newblock \doi{10.1016/j.ejc.2008.11.010}.
\newblock \msn{2514645}.

\end{thebibliography}

\def\soft#1{\leavevmode\setbox0=\hbox{h}\dimen7=\ht0\advance \dimen7
  by-1ex\relax\if t#1\relax\rlap{\raise.6\dimen7
  \hbox{\kern.3ex\char'47}}#1\relax\else\if T#1\relax
  \rlap{\raise.5\dimen7\hbox{\kern1.3ex\char'47}}#1\relax \else\if
  d#1\relax\rlap{\raise.5\dimen7\hbox{\kern.9ex \char'47}}#1\relax\else\if
  D#1\relax\rlap{\raise.5\dimen7 \hbox{\kern1.4ex\char'47}}#1\relax\else\if
  l#1\relax \rlap{\raise.5\dimen7\hbox{\kern.4ex\char'47}}#1\relax \else\if
  L#1\relax\rlap{\raise.5\dimen7\hbox{\kern.7ex
  \char'47}}#1\relax\else\message{accent \string\soft \space #1 not
  defined!}#1\relax\fi\fi\fi\fi\fi\fi}

\end{document}